\documentclass[graybox]{svmult}

\usepackage{type1cm}        % activate if the above 3 fonts are
\usepackage{makeidx}         % allows index generation
\usepackage{graphicx}        % standard LaTeX graphics tool
\usepackage{multicol}        % used for the two-column index
\usepackage[bottom]{footmisc}% places footnotes at page bottom

\usepackage{newtxtext}       % 
\usepackage{newtxmath}       % selects Times Roman as basic font

\usepackage[boxed]{algorithm2e}

\makeindex             % used for the subject index
\begin{document}

\title*{Logic-based Benders Decomposition 
	\\ for Large-scale Optimization
}
\author{J. N. Hooker}
\institute{J. N. Hooker \at Carnegie Mellon University \email{jh38@andrew.cmu.edu}
}
\maketitle

%\abstract*{Each chapter should be preceded by an abstract (no more than 200 words) that summarizes the content. The abstract will appear \textit{online} at \url{www.SpringerLink.com} and be available with unrestricted access. This allows unregistered users to read the abstract as a teaser for the complete chapter.
%Please use the 'starred' version of the \texttt{abstract} command for typesetting the text of the online abstracts (cf. source file of this chapter template \texttt{abstract}) and include them with the source files of your manuscript. Use the plain \texttt{abstract} command if the abstract is also to appear in the printed version of the book.}

\abstract{
	Logic-based Benders decomposition (LBBD) is a substantial generalization of classical Benders decomposition that, in principle, allows the subproblem to be any optimization problem rather than specifically a linear or nonlinear programming problem.  It is amenable to a wide variety large-scale problems that decouple or otherwise simplify when certain decision variables are fixed.  This chapter presents the basic theory of LBBD and explains how classical Benders decomposition is a special case.  It also describes branch and check, a variant of LBBD that solves the master problem only once.  It illustrates in detail how Benders cuts and subproblem relaxations can be developed for some planning and scheduling problems.  It then describes the role of LBBD in three large-scale case studies.  The chapter concludes with an extensive survey of the LBBD literature, organized by problem domain, to allow the reader to explore how Benders cuts have been developed for a wide range of applications.  
}

\section{Introduction}

The fundamental challenge of large-scale optimization is that its difficulty tends to increase superlinearly, even exponentially, with the size of the problem.  The challenge can often be overcome by solving the problem with a heuristic method, but only if one is willing to sacrifice optimality, or at least a proof of optimality.  If a provably optimal solution is desired, decomposition may be the only practical recourse.

The advantage of decomposition is that it breaks a problem into smaller subproblems that are easier to solve.  Due to superlinear complexity growth, solving many small subproblems can require much less computational effort than solving one large problem.  The disadvantage of decomposition is that to achieve optimality, the subproblems must somehow communicate with each other, and it may be necessary to solve them repeatedly to converge to a solution.  Nonetheless, when a problem has suitable structure, an algorithm based on decomposition can transform an intractable problem into a tractable one.

One of the best known and most successful decomposition strategies is {\em Benders decomposition}, which dates from the early 1960s \cite{Ben62}.  It was originally designed for problems that become linear programming (LP) problems, known as {\em subproblems}, when certain variables are fixed.  The duals of the subproblems are solved to obtain {\em Benders cuts}, which are constraints written in terms of the variables that were fixed.  A {\em master problem} is then solved, subject to Benders cuts generated so far, to find another set of values for these variables, whereupon the procedure repeats.  The Benders cuts exclude undesirable solutions, and the algorithm converges to a provably optimal solution under weak conditions.  The Benders approach is most attractive when the subproblem is not only linear but decouples into smaller subproblems that can be solved independently.  
%We will therefore speak of ``subproblems'' in the plural, even though in some cases there is only one subproblem.  

%The subproblems communicate with the master problem by sending it {\em Benders cuts} that are obtained from the dual solution of the subproblems.  The Benders cuts are added to the master problem, which is re-solved to obtain new values for its variables, whereupon the process repeats.  As the Benders cuts accumulate, they comprise an increasingly accurate representation of the solution space, at least in the vicinity of the optimal solution.  Under fairly weak conditions, the procedure converges finitely to a provably optimal solution.

Although classical Benders decomposition has many successful applications, its basic strategy is substantially restricted by the fact that the subproblem must be an LP problem---or a continuous nonlinear programming (NLP) problem in a 1972 extension to ``generalized'' Benders decomposition \cite{Geoffrion72}.  There are a wide range of potential applications in which the subproblem simplifies without yielding an LP or NLP problem, often by decoupling into smaller problems.  Classical Benders decomposition cannot exploit this kind of problem structure.

{\em Logic-based Benders decomposition} (LBBD), introduced in \cite{Hoo00,HooOtt03}, addresses this issue by recognizing that the classical Benders method is a actually a special case of a much more general method.  LBBD extends the underlying Benders strategy to cases in which the subproblem is an arbitrary optimization problem.  It obtains Benders cuts by solving an {\em inference dual} of the subproblem, which reduces to the LP dual when the subproblem is linear.   
%Alternatively, Benders cuts can be generated at nodes of a search tree that solves the master problem, rather than after the master problem is solved, resulting in a method known as {\em branch and check}. 

Due due its greater versatility, LBBD has a large and rapidly growing range of successful applications.  In many cases it leads to computational speedups of several orders of magnitude over the previous state of the art.  It introduces a complication, however, that is not present in classical Benders methods.  Logic-based Benders cuts must be developed anew for each problem class, while classical Benders cuts are automatically given by the LP dual of the subproblem.  This can be viewed as a drawback, but it can also be an advantage.  It may allow one to exploit the special structure of a given problem class with specially crafted Benders cuts, resulting in a effective solution method.  

{\em Branch and check}, also introduced in \cite{Hoo00}, is a variation of LBBD.  Rather than generate Benders cuts after each master problem is solved, it solves the master problem only once, by a branching method.  When a feasible solution is found in the course of branching, the resulting subproblem is solved to obtain a Benders cut that is enforced throughout the remainder of the branching search.  This method was first compared computationally with standard LBBD in \cite{Tho01}, which introduced the term ``branch and check.'' A related approach was later proposed specifically for mixed integer/linear programming (MILP) in \cite{CodFis06}, where the cuts are called {\em combinatorial Benders cuts}.

An important advantage of LBBD is that it provides a natural means to combine different kinds of problem formulations and solvers.  For example, the master problem might be amenable to an MILP formulation and solver, while the subproblem might be better suited for constraint programming (CP).  The MILP/CP combination is probably the most popular, because problems frequently decompose into an assignment problem suitable for MILP and a scheduling problem on which CP methods tend to excel \cite{BapLepNui01,Hooker07a}.  
%The combinatorial nature of the scheduling problem is no longer a barrier to generating Benders cuts.

There is an outmoded perception that Benders decomposition converges slowly and is therefore often unsuitable even for problems that have a natural decomposition.  We first remark that this perception is based on experience with classical Benders methods, not LBBD.  Even classical Benders methods have been substantially accelerated over the last two decades, using a number of devices.   An excellent survey of these improvements can be found in \cite{RahCraGen17}, which covers both classical methods and LBBD.  These authors also document an explosion of literature on Benders methods since 2000 or so, no doubt due to improvements in performance.

We begin below with an exposition of the theory behind LBBD and a precise statement of the LBBD algorithm, followed by an explanation of how classical Benders decomposition is a special case.  We then discuss branch and check, a variant of LBBD in which the master problem is solved only once, and when it is likely to be preferable to standard LBBD.  Following this is a detailed presentation of how LBBD applies to a job assignment and scheduling problem with various objective functions.  This discussion illustrates how to formulate logic-based Benders cuts for a class of problems that have perhaps benefited most frequently from LBBD to date.  It also shows how to create subproblem relaxations for this class of problems, since such relaxations are often essential to the success of LBBD.  We then briefly describe three case studies in which LBBD performed successfully in the context of large-scale optimization.  Finally, since logic-based Benders cuts are problem-specific, it can be helpful to examine previous LBBD applications in a similar problem domain, to learn how others have exploited problem structure.  Fortunately, a wide variety of LBBD applications now appear in the literature, and the chapter concludes with an extensive survey of these.

\section{Fundamentals of LBBD}

We begin by defining the {\em inference dual}, which is a basic element of LBBD.  Consider a general optimization problem $\min \{f(x) \;|\; C(x), \; x\in D\}$, 
%\label{eq:lbbd1}
%\end{equation}
%\begin{equation}
%\min \{f(x) \;|\; C(x), \; x\in D\}
%\label{eq:lbbd1}
%\end{equation}
in which $C(x)$ represents a constraint set containing variables in $x=(x_1,\ldots, x_n)$, and $D$ is the domain of $x$ (such as tuples of nonnegative reals or integers).  
%The inference dual of (\ref{eq:lbbd1}) is the problem of finding the tightest lower bound on the objective function that can be deduced from the constraints, or
The inference dual is the problem of finding the tightest lower bound $v$ on the objective function that can be deduced from the constraints, or
\begin{equation}
\max \Big\{ v \;\Big|\; C(x) \stackrel{P}{\vdash} \big(f(x) \geq v\big),  \; v\in \mathbb{R}, \; P\in {\cal P}\Big\}
\label{eq:lbbd2}
\end{equation}
Here $C(x) \stackrel{P}{\vdash} (f(x) \geq v)$ indicates that proof $P$ deduces $f(x)\geq v$ from $C(x)$.  The domain of variable $P$ is a family $\mathcal{P}$ of proofs, and a solution of the dual is a proof of the tightest bound $v$.  Thus the inference dual is always defined with respect to an inference method that determines the family of proofs in $\mathcal{P}$.   For a feasibility problem with no objective function, the dual can be viewed as the problem finding a proof $P$ of infeasibility.

In practical applications of LBBD, the dual is defined with respect to the inference method used to prove optimality (or infeasibility) when solving the subproblem.  We therefore assume that the inference dual is a strong dual: its optimal value is equal to the optimal value of the original problem.\footnote{An infeasible problem is viewed as having optimal value $\infty$ (when minimizing) or $-\infty$ (when maximizing).}
%, and the solution of the dual is the proof used to establish optimality or infeasibility.  
When the subproblem is an LP problem, the inference method is nonnegative linear combination of inequalities, and the inference dual becomes the LP dual, as we will see in the next section.

%All standard optimization duals are special cases of inference duality, including the LP, Lagrangean, surrogate, and subadditive duals \cite{Hooker12}.  

We now define LBBD, which is applied to a problem of the form 
\begin{equation}
\min \big\{ f(x,y) \;\big|\; C(x,y),\; C'(x), \; x\in D_x, \; y\in D_y \big\} \label{eq:prob}
\end{equation}
%\begin{equation}
%\min \big\{ f(x,y) \;|\; C(x,y), \; x\in D_x, \; y\in D_y \big\}
%\label{eq:lbbd3}
%\end{equation}
Fixing $x$ to $\bar{x}$ 
%in (\ref{eq:lbbd3}) 
defines the {\em subproblem} 
\begin{equation}
\min \big\{ f(\bar{x},y) \;\big|\; C(\bar{x},y),\; y\in D_y \big\}  \label{eq:sub}
\end{equation}
The inference dual of the subproblem is
\begin{equation}
\max \Big\{ v \;\Big|\; C(\bar{x},y) \stackrel{P}{\vdash} \big(f(\bar{x},y) \geq v\big),  \; v\in \mathbb{R}, \; P\in {\cal P}\Big\}  \label{eq:dual}
\end{equation}
Let $v^*$ be the optimal value of the subproblem ($\infty$ if the subproblem is infeasible), and let proof $P^*$ solve the inference dual by deducing the bound $f(\bar{x},y)\geq v^*$.  The essence of LBBD is that {\em this same proof} may deduce useful bounds when $x$ is fixed to values other than $\bar{x}$.  The term ``logic-based'' refers to this pivotal role of logical deduction.  A {\em Benders cut} $z\geq B_{\bar{x}}(x)$ is derived by identifying a bound $B_{\bar{x}}(x)$ that proof $P^*$ deduces for a given $x$.  Thus, in particular, $B_{\bar{x}}(\bar{x})=v^*$.  The cut is added to the {\em master problem}, which in iteration $k$ of the Benders procedure is
\begin{equation}
z_k = \min \big\{ z \;\big|\; C'(x);\; z \geq B_{x^i}(x),\; i=1, \ldots, k; \; x\in D_x \big\} \label{eq:master}
\end{equation}
where $x^1, \ldots, x^k$ are the solutions of the master problems in iterations $1, \ldots, k$, respectively.  

In any iteration $k$, the optimal value $z_k$ of the master problem is a lower bound on the optimal value of (\ref{eq:prob}), 
%(\ref{eq:lbbd1}), 
and the optimal value $v^*=v_k$ of the subproblem is an upper bound.  The master problem values $z_k$ increase monotonically as the iterations progress, while the subproblem values $v_k$ can move up or down.  The Benders algorithm terminates when the optimal value of the master problem equals the optimal value of the subproblem in some previous iteration.  More precisely, it terminates when $z_k=\min\{v_i\;|\;i=1, \ldots, k\}$, or when $z_k=\infty$ (indicating an infeasible problem).  
%Upon termination, an optimal solution is $(x^k,y^{i^*})$, where $y^{i^*}$ is an optimal solution of the subproblem in 
%teration $i^*$.  

At any point in the procedure, the Benders cuts in the master problem partially describe the projection of 
%(\ref{eq:lbbd3})'s 
the feasible set of (\ref{eq:prob}) onto $x$.  Even when the procedure is terminated early, it yields a lower bound $z_k$ and upper bound $\min_i\{v_i\}$ on the optimal value, as well as the best feasible solution found so far.  

A formal statement of the LBBD procedure appears as Algorithm~\ref{alg:LBBDa}.  Since $z$ is unconstrained in the initial master problem $\min\{z\;|\;C'(x), \; x\in D_x\}$, we have $z_0=-\infty$, and any feasible $x$ can be selected as the solution $x^0$.  Alternatively, we can use a ``warm start'' by generating a few Benders cuts in advance for heuristically chosen values of $\bar{x}$.

\begin{algorithm}
	%\SetAlgoLined
	$k\gets 0$,$\;$ $v_0\gets -\infty$,$\;$ $v_{\min}\gets \infty$\;
	\Repeat{$z_k = v_{\min}$} {
		\If{the master problem (\ref{eq:master}) is infeasible} {
			{\bf stop}; the original problem (\ref{eq:prob}) is infeasible\;
		}
		let $x^k$ be an optimal solution of the master problem (\ref{eq:master}) with optimal value $z_k$\;
		solve the subproblem (\ref{eq:sub}) with $\bar{x}=x^k$\;
		\If{the subproblem (\ref{eq:sub}) is unbounded} {
			{\bf stop}; the original problem (\ref{eq:prob}) is unbounded\;
		}
		$k\gets k+1$\;
		let $v_k$ be the optimal value of the subproblem (\ref{eq:sub}), where $v_k=\infty$ if (\ref{eq:sub}) is infeasible\;
		generate a Benders cut $z\geq B_{x^k}(x)$ such that $B_{x^k}(x^k)=v_k$\;
		\If{$v_k<\infty$}{
			let $y^k$ be an optimal solution of the subproblem (\ref{eq:sub})\;
			\If{$v_k<v_{\min}$} {
				$v_{\min}\gets v_k$, $\;$ $y^{\mathrm{best}}\gets y^k$\;
			}
		}
	}
	an optimal solution of the original problem (\ref{eq:prob}) is $(x,y)=(x^k,y^{\mathrm{best}})$\;
	\caption{LBBD procedure when the subproblem is an optimization problem.} \label{alg:LBBDa}
\end{algorithm}

If the subproblem is a feasibility problem with no objective function, the procedure continues until a feasible solution of the subproblem is found.  In this case, the original problem (\ref{eq:prob}) has the form
\begin{equation}
\min \big\{ f(x) \;\big|\; C(x,y),\; C'(x), \; x\in D_x, \; y\in D_y \big\} \label{eq:probFeas}
\end{equation}
and the subproblem is a constraint set
\begin{equation}
\{ C(\bar{x},y),\; y\in D_y \big\}  \label{eq:subFeas}
\end{equation}
An infeasible subproblem gives rise to a {\em feasibility cut}, which is a constraint $N_{x^k}(x)$ that is violated when $x=\bar{x}$.  The feasibility cut is added to the master problem
\begin{equation}
z_k = \min \big\{ f(x) \;\big|\; C'(x);\; N_{x^i}(x),\; i=1, \ldots, k; \; x\in D_x \big\} \label{eq:masterFeas}
\end{equation}
This version of the Benders algorithm yields no feasible solution until it terminates, but it still provides a valid lower bound $z_k$ on the optimal value in any iteration $k$.  A statement of the procedure appears as Algorithm~\ref{alg:LBBDb}.

\begin{algorithm}
	%\SetAlgoLined
	$k\gets 0$,$\;$ $\mathrm{feasible}\gets \mathrm{false}$\;
	\Repeat{$\mathrm{feasible}=\mathrm{true}$} {
		\If{master problem (\ref{eq:master}) is infeasible} {
			{\bf stop}; original problem (\ref{eq:prob}) is infeasible\;
		}
		let $x^k$ be an optimal solution of (\ref{eq:master}) with value $z_k$\;
		solve subproblem (\ref{eq:subFeas}) with $\bar{x}=x^k$\;
		$k\gets k+1$\;
		\eIf{(\ref{eq:subFeas}) is infeasible}{
			generate a feasibility cut $N_{x^k}(x)$ that is violated when $x=x^k$\;
		} {
			let $\mathrm{feasible}\gets\mathrm{true}$, and let $y^k$ be a feasible solution of (\ref{eq:subFeas})\;
		} 	
	}
	an optimal solution of the original problem (\ref{eq:prob}) is $(x,y)=(x^k,y^k)$\;
	\caption{LBBD procedure when the subproblem is a feasibility problem.} \label{alg:LBBDb}
\end{algorithm}
The simplest sufficient condition for finite convergence of LBBD is that the master problem variables have finite domains.  This is normally the case in practice, since continuous variables (if any) typically occur in the subproblem. The following is shown in \cite{HooOtt03}.

\begin{theorem}
	If the domains of the master problem variables are finite, Algorithm~\ref{alg:LBBDa} and Algorithm~\ref{alg:LBBDb} terminate after a finite number of steps. 
\end{theorem}

\section{Classical Benders Decomposition}

LBBD reduces to the classical Benders method when the subproblem is an LP problem and the inference dual is based on nonnegative linear combination and domination.  To see this, we first show that the inference dual is the classical LP dual.  Consider an LP problem 
\begin{equation}
\min\{cx \;|\; Ax\geq b, \; x\geq 0\} \label{eq:LPprob0}
\end{equation}
We suppose that an inequality $cx\geq v$ is deduced from $Ax\geq b$ when some nonnegative linear combination (surrogate) $uAx\geq ub$ of the constraint set dominates $cx\geq v$, where domination means that $uA\leq c$ and $ub\geq v$.  The inference dual maximizes $v$ subject to the condition that $cx\geq v$ can be deduced from $Ax\geq b$ and can therefore be written
\[
\max\big\{ v \;\big| \; uA\leq c, \; ub\geq v, \; u\geq 0\big\}
\]
This is equivalent to the classical LP dual $\max\{ub\;|\;uA\leq c, \; u\geq 0\}$.  If (\ref{eq:LPprob0}) 
has a finite optimal value $v^*$ and $\bar{u}$ is an optimal dual solution, we have $v^*=\bar{u}b$ by classical duality theory.  The tuple $\bar{u}$ of dual multipliers therefore encodes a proof of optimality by deducing the bound $cx\geq \bar{u}b$.

Classical Benders decomposition is applied to a problem of the form 
\begin{equation}
\min\{f(x)+cy\;|\;g(x)+Ay\geq b, \; x\in D_x, \; y\geq 0\}  \label{eq:LPprob}
\end{equation}
The subproblem is the LP problem
\[
\min\big\{f(\bar{x})+cy\;\big|\;Ay\geq b-g(\bar{x}), \; y\geq 0\big\}
\]
If the subproblem has a finite optimal value $v^*$ and $\bar{u}$ is an optimal dual solution, classical duality theory implies that $v^*=h(\bar{x}) + \bar{u}(b-g(\bar{x}))$, and the tuple $\bar{u}$ of dual multipliers therefore encodes a proof of optimality.  The essence of classical Benders decomposition is that this same tuple of multipliers (i.e., this same proof) yields a lower bound $h(x) + \bar{u}(b-g(x))$ on the optimal value of (\ref{eq:LPprob}) for any $x$.  We therefore have a Benders cut $z\geq h(x)+\bar{u}(b-g(x))$.  If the subproblem is infeasible and its dual is feasible (and therefore unbounded), the Benders cut is $\bar{u}(b-g(x))\leq 0$, where $\bar{u}$ is an extreme ray solution of the subproblem dual.

\section{Branch and Check}

{\em Branch and check} is a variation of LBBD that solves the master problem only once.  It is most naturally applied when a branching procedure solves the master problem, and the subproblem is a feasibility problem.  When a feasible solution of the master problem is encountered during the branching process, it is ``checked'' by solving the subproblem that results.  If the subproblem is infeasible, a feasibility cut is added to the master problem and enforced during the remainder of the tree search.  The algorithm terminates when the search is exhaustive.  

Branch and check can be an attractive alternative when the master problem is significantly harder to solve than the subproblem.  Under the right conditions, it can bring orders-of-magnitude speedups relative to standard LBBD.  A computational comparison of the two methods is provided in \cite{Bec10}.

Branch and check is applied to a problem in the form (\ref{eq:probFeas}).  In all applications to date, the initial master problem $\min\{f(x) \;|\; C'(x), \; x\in D_x\}$ is a mixed integer/linear programming (MILP) problem that is solved by a branch-and-cut method.  When a feasible solution $\bar{x}$ is discovered at a node of the branching tree, perhaps because $\bar{x}$ is an integral solution of the current LP relaxation, the corresponding subproblem (\ref{eq:subFeas}) is solved.  If (\ref{eq:subFeas}) is infeasible, a  cut $N_{\bar{x}}(x)$ is derived.  Since the master problem is an MILP problem, the cut must take the form of a linear inequality.  

If the subproblem is feasible, the tree search continues in the normal fashion.  If a feasibility cut is generated, the current solution $\bar{x}$ is no longer feasible because it violates the cut.  The current LP relaxation is re-solved after adding the feasibility cut, and the search again continues in the usual fashion.  The stopping condition is the same as for normal branch and cut.  At termination, the incumbent solution (if any) is optimal for (\ref{eq:probFeas}) because it defines a feasible subproblem.  

Branch and check is not a special case of branch and cut, because its feasibility cuts are obtained in a different fashion.  Unlike the cuts used in branch-and-cut methods, they are not are valid for the MILP problem being solved.  They are based on a subproblem constraint set that does not appear in the MILP problem.  They intermingle with standard cuts during the tree search but have different origins.

MILP solvers typically use a primal heuristic to generate feasible solutions at the root node of the search tree, and perhaps at other nodes.  These feasible solutions can be used to obtain additional feasibility cuts, sometimes to great advantage.  Another practical consideration is that branch and check requires modification of the code that solves the MILP master problem.  This contrasts with standard LBBD, which can use an off-the-shelf method.  Branch and check can therefore take longer to implement.  

In an extended form of branch and check, there is no separate subproblem, but partial solutions found during the branching process are checked for feasibility.  At any given node of the branching tree, the variables have been fixed so far can be treated as the master problem variables.  Their values are checked for feasibility by solving the subproblem that remains after they are fixed.  If infeasibility is verified, the dual solution of the subproblem can form the basis of a Benders cut.  Such a method can be regarded as a branch-and-check algorithm with a dynamic partition of the master problem and subproblem variables.  

Interestingly, this is the most popular scheme used in state-of-the-art satisfiability (SAT) solvers, where it is known as conflict-directed clause learning \cite{BeaKauSab03}.  Partial solutions are not necessarily obtained by straightforward branching, but their feasibility is nonetheless checked by solving a subproblem in the form of an implication graph.  If infeasibility is detected, a dual solution is derived by identifying a unit resolution proof of infeasibility represented by a conflict graph within the implication graph.  A conflict clause (Benders cut) is obtained from a certain kind of partition of the conflict graph.  Modern SAT solvers can handle industrial instances with well over a million variables.  Their extraordinary efficiency is due in large part to clause learning, which is basically a form of branch and check.

\section{Example: Job Assignment and Scheduling}

An initial example will illustrate several practical lessons for applying LBBD:
\begin{itemize}
	
	\item The master problem and subproblem are often best solved by different methods that are suited to the structure of the two problems.
	
	\item Often the subproblem solver does not provide easy access to its proof of optimality (or infeasibility), and Benders cuts must be based on dual information that is obtained indirectly.  
	
	\item It is usually important to include a relaxation of the subproblem in the master problem, expressed in terms of master problem variables.
	
\end{itemize}
%Methods for relaxing the subproblem are discussed in the next section.

The example problem is as follows \cite{Hooker07}.  Jobs $1,\ldots,n$ must be assigned to facilities $1,\ldots, m$, and the jobs assigned to each facility must be scheduled.  Each job $j$ has processing time $p_{ij}$ on facility $i$, release time $r_j$, and due date $d_j$.  The facilities allow {\em cumulative scheduling}, meaning that jobs can run in parallel so long as the total rate of resource consumption does not exceed capacity.  Job $j$ consumes resources at the rate $c_j$, and facility $i$ has a resource capacity of $C_i$.  If $c_j=C_i=1$ for all $j$ and $i$, we have a {\em disjunctive scheduling} problem in which jobs run one at a time without overlap.  Various objectives are possible, such as minimizing makespan, processing cost, the number of late jobs, or total tardiness.

The problem decomposes naturally.  If the master problem assigns jobs to processors, and the subproblem schedules jobs, the subproblem decouples into a separate scheduling problem for each facility.  Given that the scheduling component of the problem is the most difficult to scale up, this is a substantial benefit because it breaks up the scheduling problem into smaller pieces.  Such a decomposition also allows appropriate solution methods to be applied to the master problem and subproblem.  MILP tends to be very effective for assignment-type problems, while constraint programming (CP) is often the method of choice for scheduling problems.  

Since the master problem is to be solved as a MILP problem, we formulate it with 0--1 variables and linear inequality/equality constraints. Let $x_{ij}$ take the value 1 when job $j$ is assigned to facility $i$.  If we choose to minimize makespan $M$ (the time at which the last job finishes), the master problem (\ref{eq:prob}) is
\begin{equation}
\min \Big\{ M \;\Big|\;M\geq M_i, \;\mbox{all $i$}; \; \sum_i x_{ij} = 1, \;\mbox{all $j$}; \;  \mbox{Benders cuts}; \; x_{ij}\in\{0,1\},\;\mbox{all $i,j$} \Big\}
\label{eq:sched1}
\end{equation}
The variable $M_i$ is the makespan on facility $i$ and will appear in the cuts.  Let $\bar{x}_{ij}$ be the solution of the master problem, and $J_i=\{i\;|\; \bar{x}_{ij}=1\}$ the set of jobs assigned to processor $i$.  If variable $s_j$ is the start time of job $j$, the subproblem for each facility $i$ can be given the CP formulation
\begin{equation}
\begin{array}{l}
\min \Big\{M_i \;\Big|\; 
M_i\geq s_j+p_{ij},\; r_j \leq s_j \leq d_j-p_{ij}, \;\mbox{all $j\in J_i$}; \\
\hspace{25ex} \mathrm{cumulative}\big(s(J_i),p_i(J_i),c(J_i),C_i\big)\Big\}
\end{array}
\label{eq:sched2}
\end{equation}
where $s(J_i)$ is the tuple of variables $s_j$ for $j\in J_i$, and similarly for $p_i(J_i)$ and $c(J_i)$.  The {\em cumulative constraint},
%$\mathrm{cumulative}(s(J_i),p(J_i),c(J_i),C_i)$ 
a standard global constraint in CP, requires that the jobs running at any one time have resource consumption rates that sum to at most $C_i$.

Benders cuts can be obtained as follows.  Let $M_i^*$ be the optimal makespan obtained for facility $i$.  We wish to obtain a Benders cut $M_i\geq B_{i\bar{x}}(x)$ for each facility $i$ that bounds the makespan for any assignment $x$, where $B_{i\bar{x}}(\bar{x})=M_i^*$.  Ideally, we would examine the proof of the optimal value $M_i^*$ obtained by the CP solver and determine what kind of bound this same proof deduces when a different set of jobs is assigned to facility $i$. 
%The proof takes the form of a branching tree, plus such domain reduction procedures such as edge finding, extended edge finding, not-first/not-last rules, timetabling, and energetic reasoning \cite{}.  
However, the solver typically does not provide access to a proof of optimality.
%which could be quite complicated in any case if the branching tree is large.  

We must therefore rely on information about the proof obtained indirectly.  The most basic information is which job assignments appear as premises in the proof.  If we could find a smaller set $J'_i\subset J_i$ that contains the jobs whose assignments serve as premises, we could write a cut
\begin{equation}
M_i\geq M_i^*\Big(1-\sum_{j\in J'_i} (1-x_j)\Big) \label{eq:makespanNogood}
\end{equation}
that imposes the bound $M_i^*$ whenever the jobs in $J'_i$ are all assigned to facility $i$.  One way to obtain $J'_i$ is to tease out the structure of the proof by removing jobs from $J_i$ one at a time and re-solving the scheduling problem until the minimum makespan drops below $M^*_i$.  The last set of jobs for which the makespan is $M^*_i$ becomes $J'_i$.  This simple procedure can be quite effective, because in many applications the individual scheduling problems can be re-solved very rapidly.  We will refer to cuts like (\ref{eq:makespanNogood}) as {\em strengthened nogood cuts}, a term that originates with analogous feasiblity cuts.
%that are based a set of assigned jobs that is ``no good'' because it results in infeasiblity.  
A cut of the form (\ref{eq:makespanNogood}) should be generated and added to the master problem for each facility $i$.  

A weakness of strengthened nogood cuts is that they provide no useful bound when not all jobs in $J'_i$ are assigned to facility $i$.  This weakness can often be overcome by using {\em analytical Benders cuts} that are based on an analysis of the subproblem structure.  For example, if all the release times are the same, we can prove a lemma that gives rise to more useful Benders cuts.  We give the proof from \cite{Hooker07} to illustrate the type of reasoning that is often employed in the derivation of Benders cuts.

\begin{lemma} \label{le:makespan}
	Suppose all release times $r_j=0$, and $M^*_i$ is the optimal makespan on facility $i$ when the jobs in $J'_i$ are assigned to it.  If the jobs in $S\subseteq J'_i$ are removed from facility $i$, the optimal makespan $M_i$ of the resulting problem satisfies
\begin{equation}
M_i \geq M^*_i - \max_{j\in S} \{d_j\} + \min_{j\in S} \{d_j\} - \sum_{j\in S} p_{ij}
\label{eq:sched10}
\end{equation}	
\end{lemma}

\begin{proof}
Starting with the optimal schedule that yields makespan $M_i$, we can create a schedule for $J'_i$ with makespan $M_i+\sum_{j\in S}p_{ij}$ by scheduling the jobs in $S$ consecutively and contiguously, beginning at time $M_i$.  We consider two cases:
\[
\mbox{(a)} \; M_i+\sum_{j\in S}p_{ij}\leq \min_{j\in S} \{d_j\}, \;\;\;\mbox{(b)} \; M_i+\sum_{j\in S}p_{ij} > \min_{j\in S} \{d_j\}
\]
In case (a), the schedule is feasible, and we have $M^*_i\leq M_i+\sum_{j\in S}p_{ij}$ because $M^*_i$ is optimal.  But this implies (\ref{eq:sched10}).  In case (b), we add $\max_{j\in S}\{d_j\}$ to both sides of (b) and rearrange terms to obtain
\[
M_i+\sum_{j\in S}p_{ij}+\max_{j\in S}\{d_j\} - \min_{j\in S}\{d_j\} > \max_{j\in S} \{d_j\} \geq M^*
\]
where the second inequality is due to the fact that $M^*_i$ results from a feasible solution.  This again implies (\ref{eq:sched10}). 
\end{proof}

To obtain an analytic Benders cut from Lemma~\ref{le:makespan}, we interpret $S$ as the set of jobs in $J'_i$ that are no longer assigned to facility $i$ in subsequent Benders iterations; that is, the jobs $j$ for which $x_{ij}=0$.  Thus (\ref{eq:sched10}) implies the cut 
\begin{equation}
M_i \geq M^*_i - \sum_{j\in J'_i} p_{ij}(1-x_{ij}) + \max_{j\in J'_i} \{d_j\} - \min_{j\in J'_i} \{d_j\} \label{eq:makespanAnalytic}
\end{equation}
because $\max_{j\in J_i}\{d_j\}\geq \max_{j\in S}\{d_j\}$ and $\min_{j\in J_i}\{d_j\}\leq \max_{j\in S}\{d_j\}$.  A cut of this form is generated for each facility $i$ and added to the master problem.  These cuts should be used alongside the strengthened nogood cuts (\ref{eq:makespanNogood}), which impose a tighter bound $M^*_i$ when no jobs are removed from facility $i$ and the deadlines differ.

A similar line of argument establishes analogous cuts when the jobs have different release times but no deadlines, an assumption perhaps better suited to minimum makespan problems:
\begin{equation}
M_i \geq M^*_i - \sum_{j\in J'_i} p_{ij}(1-x_{ij}) - \max_{j\in J'_i} \{r_j\} + \min_{j\in J'_i} \{r_j\} \label{eq:makespanAnalytic2}
\end{equation} 
These cuts should also be used alongside the strengthened nogood cuts (\ref{eq:makespanNogood}).

The Benders cuts are similar for other objective functions.  If the objective is to minimize assignment cost, we let $c_{ij}$ be the cost of assigning job $j$ to facility $i$.  The master problem becomes 
\[
\min \Big\{ \sum_j c_{ij}x_{ij}\;\Big|\;\sum_i x_{ij} = 1, \;\mbox{all $j$}; \;  \mbox{Benders cuts}; \; x_{ij}\in\{0,1\},\;\mbox{all $i,j$} \Big\}
\]
and the subproblem for facility $i$ is the feasibility problem
\[
\Big\{r_j \leq s_j \leq d_j-p_{ij}, \;\mbox{all $j\in J_i$};\;\mathrm{cumulative}\big(s(J_i),p_i(J_i),c(J_i),C_i\big)\Big\}
\]
Strengthened nogood cuts take the form $\sum_{j\in J'_i} (1-x_j)\geq 1$ and are derived in a similar fashion as the makespan cuts.

If the objective is to minimize total tardiness, the deadlines become due dates, and the master problem is
\[
\min \Big\{ \sum_i T_i \;\Big|\;\sum_i x_{ij} = 1, \;\mbox{all $j$}; \;  \mbox{Benders cuts}; \; x_{ij}\in\{0,1\},\;\mbox{all $i,j$} \Big\}
\]
where the variable $T_i$ is the tardiness on facility $i$ and will appear in the Benders cuts.  The subproblem for facility $i$ is 
\[
\min \Big\{\sum_{j\in J_i} (s_j+p_{ij}-d_j)^+ \;\Big|\; 
r_j \leq s_j, \;\mbox{all $j\in J_i$};\;
\mathrm{cumulative}\big(s(J_i),p_i(J_i),c(J_i),C_i\big)\Big\}
\]
where $(\alpha)^+=\max\{0,\alpha\}$.  There are various schemes for deriving strengthened nogood cuts.  One that has been used successfully \cite{Hooker07} goes as follows.  Let $T^*_i(J)$ be the minimum tardiness on facility $i$ when the jobs in $J$ are assigned to it, so that $T^*_i(J_i)$ is the minimum tardiness under the current assignment $J_i$.  Let $Z_i$ be the set of jobs in $J_i$ that can be removed one at a time, with all other jobs remaining, without reducing the minimum tardiness.  Thus $Z_i=\{j\in J_i\;|\; T^*_i(J_i\setminus\{j\})=T^*_i(J_i)\}$.  Then we have the cut 
%Finally, let $T^o_i= T_i(J_i\setminus Z_i)$  
\[
T_i \geq T^*_i(J_i\setminus Z_i)  \Big(1 - \hspace{-1ex} \sum_{j\in J_i\setminus Z_i} \hspace{-1ex}(1-x_{ij}) \Big), \;\; T_i\geq 0
\]
This cut should be used alongside the cut
\[
T_i \geq T^*_i(J_i)  \Big(1 - \hspace{0ex} \sum_{j\in J_i} \hspace{0ex}(1-x_{ij}) \Big), \;\; T_i\geq 0
\]
to obtain a tighter bound $T^*(J_i)$ when no jobs are removed from facility $i$.

Analytical Benders cuts for the minimum tardiness problem are analogous to those for the minimum makespan problem, although the derivation is somewhat more involved.  
%When the rlease times are all zero and the deadlines are all the same, we have the cut
%\[
%T_i \geq M^*_i - \sum_{j\in J_i} p_{ij}(1-x_{ij})
%\]
They have the form
\[
T_i \geq M^*_i - \Big( \sum_{j\in J_i} p_{ij}(1-x_{ij}) + 
\max_{j\in J_i} \{d_j\} - \min_{j\in J_i} \{d_j\} \Big)
\]
where $M^*_i$ is the minimum makespan on facility $i$ for assignment $J_i$, a quantity that must be computed separately from the minimum tardiness.  For reasons explained in Section 6.15.5 of \cite{Hooker12}, the analytical cuts are weak for cumulative scheduling but are more effective for the special case of disjunctive scheduling.

\section{Relaxing the Subproblem}

Past experience with LBBD has shown that success often depends on the presence of a subproblem relaxation in the master problem.  It is not the typical sort of relaxation, because it is expressed in terms of the master problem variables rather than the subproblem variables.  Nonetheless, a suitable relaxation is often evident based on the structure of the subproblem.  This is illustrated here for the job assignment and scheduling problem of the previous section, using various objective functions \cite{Hooker07}.

When the objective is to minimize assignment cost, a simple {\em time window relaxation} can be very effective.  Let the {\em energy} consumed by job $j$ be $p_jc_j$.  Then it is clear that the total energy consumed by jobs that run in a given time interval $[t_1,t_2]$ can be no greater than the energy $C_i(t_2-t_1)$ that is available during that interval.  This gives rise to a simple valid inequality for facility $i$:
\[
\sum_{j\in J(t_1,t_2)}\hspace{-2ex}  p_{ij}c_{ij}x_{ij} \leq C_i(t_2-t_1)
\]
where $J(t_1,t_2)$ is the set of jobs with time windows in the interval $[t_1,t_2]$.  We will refer to this inequality as $R_i[t_1,t_2]$.  We can add a relaxation of the subproblem to the master problem by augmenting the master problem with the inequalities $R_i[r_j,d_{j'}]$ for each $i$ and each distinct pair $[r_j,d_{j'}]$ of release times and deadlines.  Actually, we can omit many of these inequalities because they are dominated by others.  Let the {\em tightness} of an inequality $R_i(t_1,t_2)$ be 
\[
\theta_i(t_1,t_2)= (1/C_i) \hspace{-1.5ex} \sum_{j\in J(t_1,t_2)} \hspace{-1.7ex} p_{ij}c_{ij} - t_2 + t_1
\]
Then the following lemma can be used to eliminate redundant inequalities:
\begin{lemma}
	Inequality $R_i[t_1,t_2]$ dominates $R_i[u_1,u_2]$ whenever $[t_1,t_2]\subseteq [u_1,u_2]$ and $\theta_i(t_1,t_2) \geq \theta_i(u_1,u_2)$.
\end{lemma}

In a minimum makespan problem, we can use similar reasoning to bound the makespan.  Let $R_i(t)$ be the inequality
\[
M_i \geq t + (1/C_i) \hspace{-1.5ex} \sum_{j\in J(t,\infty)} \hspace{-1.5ex} p_{ij}c_{ij}x_{ij}
\]
We can add inequalities $R_i(r_j)$ to the master problem for each distinct release time $r_j$.  Again, some of these may be redundant.

The minimum tardiness problem calls for less obvious relaxation schemes.  Two have been derived, the simpler of which is a time-window relaxation based on the following.
\begin{lemma}
	If jobs $1,\ldots,n$ are scheduled on a single facility $i$, the total tardiness is bounded below by
	\[
	\Big( (1/C_i)\hspace{-2ex} \sum_{j\in J(0,d_k)} \hspace{-2ex} p_{ij}c_{ij} - d_k \Big)^+
	\]
	for each $k=1,\ldots,n$.
\end{lemma}
This yields the following valid inequalities for each facility $i$:
\[
T_i \geq (1/C_i) \hspace{-2ex} \sum_{j\in J(0,d_k)} \hspace{-2ex} p_{ij}c_{ij}x_{ij} - d_k,\; k=1,\ldots,n
\]
These inequalities can be added to the master problem, along with $T_i\geq 0$, for each $i$.  To state the second relaxation, let $\pi_i$ be a permutation that orders jobs by increasing energy on facility $i$, so that $p_{i\pi_i(1)}c_{i\pi_i(1)}\leq \cdots\leq p_{i\pi_i(n)}c_{i\pi_i(n)}$.  We have
\begin{lemma}
	If jobs $1,\ldots,n$ are scheduled on a single facility $i$ and are indexed so that $d_1\leq\cdots\leq d_n$, the total tardiness is bounded below by $\sum_{k=1}^n \hat{T}_k$, where
	\[
	\hat{T}_k = \Big( (1/C_i) \sum_{j=1}^k p_{i\pi_i(j)} c_{i\pi_i(j)} - d_k \Big)^+, \;\; k=1, \ldots, n
	\]
\end{lemma}
This leads to a relaxation consisting of the inequality $T_i\geq \sum_{k=1}^n \hat{T}_{ik}$ for each $i$, as well as the inequalities $\hat{T}_{ik}\geq 0$ and
\[
\hat{T}_{ik} \geq (1/C_i) \sum_{j=1}^k p_{i\pi_i(j)} c_{i\pi_i(j)} x_{i\pi_i(j)} - d_k - U_{ik}(1-x_{ik})
\]
for each $i$ and $k=1,\ldots,n$.  Here $U_{ik}$ is a big-$M$ term that can be defined
\[
U_{ik} = \sum_{j=1}^k p_{i\pi_i(j)} c_{i\pi_i(j)}  - d_k
\]
The cuts are valid even when $U_{ik}<0$.

\section{Large-Scale Case Studies}

In this section we briefly highlight three case studies that illustrate how LBBD can succeed in large-scale settings.  One is a massive optimization problem associated with an incentive auction conducted by the U.S. Federal Communications Commission (FCC).  The team that designed the solution procedure received the prestigious Franz Edelman Award from INFORMS (Institute for Operations Research and the Management Sciences) in 2018.  A second case study illustrates how LBBD can scale up by using approximate solutions of the master problem and subproblem.  A third shows how LBBD can be of value even when the problem does not naturally decompose.  The reader is referred to the original papers for details regarding the models and solution methods.
\bigskip

\noindent
{\em Frequency Spectrum Allocation}
\medskip

\noindent
The FCC incentive auction was designed to reallocate parts of the frequency spectrum to television broadcasters and wireless providers, due to growing demand from the latter.  Wireless providers offered bids to TV stations for additional bandwidth.  After the auction was conducted, an optimization problem was solved to reallocate the spectrum \cite{Hof18}.  The smaller TV band that remained was reallocated to stations so as to minimize interference, and successful wireless bidders were assigned frequencies in an enlarged wireless band.  The problem was formulated for nearly 3000 U.S. and Canadian stations and initially contained some 2.7 million pairwise interference restrictions, as well as many additional constraints.  Stations in congested areas were allocated portions of the wireless band when necessary to reduce interference.  

The overall solution algorithm was an LBBD procedure in which the master problem allocated frequencies to wireless providers and certain stations in the wireless band, and the subproblem attempted to find a feasible packing of the TV band for the remaining stations.  The problem was solved to optimality. 
\bigskip

\noindent
{\em Suboptimal Solution of Master Problem and/or Subproblem}
\medskip

\noindent
The performance of LBBD can often be accelerated by solving the master problem, or even the subproblem, only approximately.  Suboptimal solution of the master problem is a well-known and often used strategy, because only feasible solutions of the master problem are required to obtain Benders cuts.  Of course, the optimal values obtained from the master problem are no longer valid lower bounds.  To obtain a provably optimal solution of the original problem, the master problem must be solved to optimality in the latter stages of the Benders procedure.  

Supoptimal solution of the subproblem is a more difficult proposition, because it can result in non-valid Benders cuts.  This possibility was investigated for classical Benders decomposition in \cite{ZakPhiRya00}, where it is assumed that a dual feasible solution is available for an LP subproblem that is not solved to optimality, as for example when using an interior-point method.   More relevant here is an application to LBBD in \cite{RaiBauHu14,RaiBauHu15}, where dramatic speedups were obtained for a vehicle routing problem by solving the subproblem and possibly the master problem with metaheuristics.  This sacrifices optimality but yields significantly better solutions, in much less time, than terminating an exact LBBD algorithm prematurely.  This study also found ways, based on specific problem structure, to improve the accuracy of previously-solved subproblems using information obtained from approximate solution of the current subproblem.  

Another possible strategy, not employed in \cite{RaiBauHu14,RaiBauHu15}, is to solve the inference dual of the subproblem directly by searching for a proof of optimality, and then terminating the search prematurely.  The resulting bound is not optimal but can serve as the basis of a valid Benders cut.  To guarantee convergence to an optimal solution, the subproblem dual must at some point be solve to optimality.  One general approach to solving the inference dual directly is given in \cite{BenHoo19}, where branching is interpreted as a solution method for the inference dual and is managed accordingly.  
\bigskip

\noindent{\em No Natural Decomposition}
\medskip

\noindent
Finally, a problem need not decompose naturally to benefit from LBBD.  This is demonstrated in \cite{CobHoo10,CobHoo13}, which solves a simple single-machine scheduling problem with time windows, but with many jobs and long time horizons.  To decompose the model, the time horizon is divided into segments.  The master problem decides in which segment(s) a job is processed, and the subproblem decouples into a scheduling problem for each segment.  Because a job can overlap two or more segments, the decomposition might be viewed as unnatural, and in fact the master problem and analytic Benders cuts are quite complex and tedious to formulate.  

However, modeling complexity does not necessarily imply computational complexity.  It was found that LBBD is much faster than stand-alone CP and MILP on minimum cost and makespan instances.  Nearly all instances were intractable for CP, and many for MILP, while only one was intractable for LBBD.  The LBBD advantage was more modest on minimum tardiness instances.  Interestingly, CP solved a few of the instances in practically zero time (presumably because the arrangement of time windows permitted effective propagation), but it timed out on the remaining instances.

\section{Survey of Applications}

As noted earlier, logic-based Benders cuts must be developed for each class of problems.  This may require ingenuity but affords an opportunity to exploit problem structure and design a superior solution algorithm.  Fortunately, there is a sizeable LBBD literature that describes how Benders cuts can be designed for particular problem classes.  Examination of previous work in an application domain similar to one's own may suggest effective cuts as well as subproblem relaxations.  To this end, we survey a variety of LBBD applications.
\bigskip

\noindent{\em Task Assignment and Scheduling}
\medskip

\noindent
The assignment and scheduling problem discussed above is further studied in \cite{CirCobHoo13,CirCobHoo16}, where updated computational testing found that LBBD remains orders of magnitude faster than the latest MILP technology, with the advantage over CP even greater.  Similar assignment and scheduling problems are solved in \cite{ChuXia05,Hoo05b,Hooker06,TraAraBec16}.  LBBD models having basically the same structure have been applied to steel production \cite{GolSch18,HarGro01}, concrete delivery \cite{KinTri14}, batch scheduling in chemical plants \cite{HarGro02,Mar06,MarGro04,Tim02}, resource scheduling with sequence-dependent setups \cite{TraBec12}, and computer processor scheduling \cite{BenBerGueMil05,BenLomManMilRug08,BenLomMilRug11,CamHlaDepJusTri04,EmeTheAleVor15,HlaCamDepJus08,LiuGuXuWuYe11,LiuYuaHeGuLiu08,LomMil06,LomMilRugBen10,RugGueBerPolMil06,SatRavKeu07}.
\bigskip

\noindent
{\em Vehicle Routing}
\medskip

\noindent
LBBD has been applied to a number of vehicle routing problems, most of which decompose into vehicle assignment and routing components.  The latter include capacitated vehicle routing \cite{RaiBauHu14,RaiBauHu15,RiaSeaWigLen13,SarWigCar13}, dispatching and routing of automated guided vehicles \cite{CorLanRou04,NisHirGro11}, dial-a-ride problems \cite{RieRai18}, 
and a senior door-to-door transportation problem (on which pure CP ``surprisingly'' performs better than LBBD) \cite{LiuAleBec18}.  In other solution approaches, the master problem selects markets to visit in the traveling purchaser problem \cite{BooTraBec16}, 
%flight path planning \cite{KnuChiLar17}
finds initial routes in a traffic diversion problem \cite{XiaEreWal04}, and assigns jobs to cranes in yard crane dispatching and scheduling problems \cite{NosBriPes18,Dij15}.  Additional LBBD applications include search and rescue \cite{RaaMolZsiPic16},
coordinating vessels for inter-terminal transport \cite{LiNegLod16,LiNegLod17}, and 
maritime traffic management \cite{AguAksLau18}.
\bigskip

\noindent
{\em Shop, Factory, and Employee Scheduling}
\medskip

\noindent
LBBD applications to shop and factory scheduling include aircraft repair shop scheduling \cite{BajBec13},
job shop scheduling with human resurce constraints \cite{GuyLemPinRiv14},
permutation flowshop scheduling with time lags \cite{HamLou13,HamLou15},
one-machine scheduling problems \cite{CobHoo13,Sad04,Sad08}, and flowshop planning and scheduling with deteriorating machines \cite{BajBec15}.  There is also an application to feature-based assembly planning \cite{KarKovVan17}.  Employee scheduling applications include 
shift selection and task sequencing \cite{BarCohGus10}, 
multi-activity shift scheduling \cite{SalWal12}, shift scheduling with fairness constraints \cite{DoiNis14}, railway crew rostering with fairness constraints \cite{NisSugInu14}, 
and a multi-activity tour scheduling problem that integrates shift scheduling with days-off planning \cite{ResGenRou18}. 
\bigskip

\noindent
{\em Other Scheduling and Logistics Problems}
\medskip

\noindent
LBBD has been applied to a variety of additional scheduling and logistics problems.  In the transportation logistics domain, they include food distribution \cite{SolSchGho14}, bicycle sharing \cite{KloPapHuRai15,KloRai17}, lock scheduling \cite{VerKinDecBer15}, and supply chain scheduling \cite{TerDogOzeBec12}.  Other applications are project scheduling \cite{KafGanSchSonStu17}, robust call center scheduling \cite{CobHecHooSch14}, task scheduling for satellites \cite{ZhuZhaQiuLi12},
course timetabling \cite{CamHebOsuPap12},
wind turbine maintenance scheduling \cite{FroGenMenPinRou17},
queuing design and control \cite{TerBecBro07,TerBecBro09},  service restoration planning for infrastructure networks \cite{GonLeeMitWal09}, and 
sports scheduling \cite{Che09,Ras08,RasTri07,TriYil07,TriYil11}.
\bigskip

\noindent
{\em Health-Related Applications}
\medskip

\noindent
LBBD applications in the rapidly growing healthcare field include operating room scheduling \cite{Luo15,RosAleUrb15,RosLuoAleUrb17,RosLuoAleUrb17a}, outpatient scheduling \cite{RiiManLam16}, and
home health care routing and scheduling \cite{CirHoo12,HecHoo16,HecHooKim19}.  The study reported in \cite{HecHooKim19} is a case in which branch and check substantially outperforms standard LBBD due to rapid solution of the subproblems relative to the master problem.
\bigskip

\noindent
{\em Facility Location}
\medskip

\noindent
Some location problems addressed by LBBD are plant location \cite{FazBec12}, inventory location \cite{WheGzaJew15}, 
stochastic warehouse location \cite{TarArmMig06}, location-allocation problems \cite{FazBec09},
and facility location and fleet management \cite{FazBerBec13}.
\bigskip

\noindent
{\em Network Design}
\medskip

\noindent
Network design applications include green wireless local area network design \cite{GenGarNenScuTav13,GenScuGarNenTav16},
transport network planning for postal services \cite{PetTri09},
broadcast domination network design \cite{SheSmi11}, and the
edge partition problem in optimal networks \cite{TasSmiAhmSch09}.  Yet another employment of LBBD is to solve the 
minimum dominating set problem \cite{GenLucCunSim14}, which is a key element of a variety of network design problems.
\bigskip

\clearpage
\noindent
{\em Other Applications}
\medskip

\noindent
LBBD has proved useful in a number of additional domains, both practical and algorithmic.  Practical applications include capacity planning \cite{GavMilOsuHol12}, 
 logic circuit verification \cite{HooYan95},
 template design \cite{TarArmMig06}, 
 strip packing \cite{CotDelIor14,MasRai15},
 orthogonal stock cutting \cite{DelIorMar17},
 robust scheduling \cite{CobHecHooSch14}, and robust optimization \cite{AssNorSanAnd17,AssNorSanAnd17a}.
 Interestingly, LBBD can also play a role in the solution of abstract problem classes, such as optimal control \cite{BorSad09}, quadratic programming \cite{BaiMitPan12,HuMitPan12,HuMitPanBenKun08}, 
chordal completion \cite{BerRag15}, 
linear complementarity \cite{HuMitPanBenKun08}, modular arithmetic \cite{KafGanSchSonStu17}, the operator count problem in automated planning algorithms \cite{DavPeaStuLip01}, and propositional satisfiability (SAT) \cite{BacDavTsiKat14}.  A hitting set method that has been successfully applied to the maximum satisfiability problem (MAXSAT) is a special case of LBBD \cite{DavBac13,DavBac11}.
\bigskip

\section{Concluding Remarks: Implementation}

One impediment to the use of LBBD may be the lack of an implementation in off-the-shelf software.  The Benders cuts must be designed by hand, and the communication between master problem and subproblem carried out by special-purpose code.  Yet solution of a large-scale problem is typically far from straightforward by any method, even using a powerful MILP or SAT solver.  An MILP model must often be carefully written or reformulated to make it tractable for a solver, and formulation of problems for a SAT solver is even more challenging.  Of course, many problems are beyond the capability of a stand-alone solver, regardless of how they are formulated.

Actually, LBBD has recently been automated in the MiniZinc modeling system \cite{DavGanStu17}.  The system chooses the decomposition and Benders cuts rather than relying on the user to do so.  This is a convenience but may result in a less effective realization of LBBD.  The ability of LBBD to benefit from user insight is a substantial advantage, since humans are much better at pattern recognition, and therefore at discerning problem structure, than machines.  There are also modeling systems that can facilitate the implementation of logic-based Benders, such as IBM's OPL Studio, the Mosel development environment, and the general-purpose solver SIMPL \cite{YunAroHoo10}. A survey of software tools for implementing LBBD and other hybrid methods, if somewhat dated, can be found in \cite{Yun11}.

As large-scale applications proliferate in our age of big data, and decomposition methods are increasingly called upon, it is likely that tools for their implementation will become increasingly powerful and make the application of methods like LBBD more routine.

%\input{references}

%\bibliographystyle{plain}
%\bibliography{biblio.bib,jnh2.bib}

\end{document}